\documentclass[a4paper]{article}
\usepackage{amssymb}
\usepackage{amsmath}
\usepackage{mathrsfs}
\usepackage{amsfonts}
\usepackage{color}
\usepackage{vmargin}
\usepackage{amsthm}

\long\def\symbolfootnote[#1]#2{\begingroup%
\def\thefootnote{\fnsymbol{footnote}}\footnote[#1]{#2}\endgroup}

\setmarginsrb{20mm}{20mm}{20mm}{20mm}{10mm}{10mm}{10mm}{10mm}

\newtheoremstyle{remark}
  {}{}{}{}{\bfseries}{.}{.5em}{{\thmname{#1 }}{\thmnumber{#2}}{\thmnote{ (#3)}}}
\theoremstyle{remboldstyle}

\newtheorem*{theorem*}{Theorem}
\newtheorem{theorem}{Theorem}
\newtheorem{lem}{Lemma}
\newtheorem{cor}{Corollary}

\theoremstyle{definition}
\newtheorem{defn}[theorem]{Definition}
\theoremstyle{remark}
\newtheorem{rem}[theorem]{Remark}
\newtheorem{ex}{Example}

\newtheorem{opprb}{\bf Open problem}

\def\esssup{\operatornamewithlimits{ess\,sup}}
\def\essinf{\operatornamewithlimits{ess\,inf}}
\def\px{p(x)}
\def\p{{p(\cdot)}}

\def\PPln{\mathcal{P}^{\log}}

\newcommand{\R}{{\mathbb R}}
\newcommand{\Rn}{{\mathbb R}^n}

\newcommand{\Om}{\Omega}
\newcommand{\A}{{\mathcal A}}
\newcommand{\B}{{\mathcal B}}

\def\div{{\rm div}}

\def\p{{p(\cdot)}}
\def\px{p(x)}

\def\qx{q(x)}

\newcommand{\kom}[1]{}
\renewcommand{\kom}[1]{{\bf [#1]}}

\definecolor{blau}{rgb}{0.1,0.0,0.9}

\newcounter{komcounter}
\numberwithin{komcounter}{section}

\def\vint{\mathop{\mathchoice%
          {\setbox0\hbox{$\displaystyle\intop$}\kern 0.22\wd0%
           \vcenter{\hrule width 0.6\wd0}\kern -0.82\wd0}%
          {\setbox0\hbox{$\textstyle\intop$}\kern 0.2\wd0%
           \vcenter{\hrule width 0.6\wd0}\kern -0.8\wd0}%
          {\setbox0\hbox{$\scriptstyle\intop$}\kern 0.2\wd0%
           \vcenter{\hrule width 0.6\wd0}\kern -0.8\wd0}%
          {\setbox0\hbox{$\scriptscriptstyle\intop$}\kern 0.2\wd0%
           \vcenter{\hrule width 0.6\wd0}\kern -0.8\wd0}}%
          \mathopen{}\int}
\usepackage{amsmath,amssymb}
\usepackage{amssymb,amsthm}

\title{\bf The Liouville theorems for elliptic equations with nonstandard growth}
\author{
Tomasz Adamowicz{\small{$^1$}}
\\
\it\small Institute of Mathematics of the Polish Academy of Sciences \\
\it\small 00-956 Warsaw, Poland\/{\rm ;}
\it\small T.Adamowicz@impan.pl
\\
\\
Przemys{\l}aw G\'{o}rka
\\
\it\small Department of Mathematics and Information Sciences,
\it\small Warsaw University of Technology,\\
\it\small Ul. Koszykowa 75, 00-662 Warsaw, Poland\/{\rm ;}
\it\small P.Gorka@mini.pw.edu.pl
}
\date{}

\begin{document}

\maketitle

\footnotetext[1]{T. Adamowicz was supported by a grant of National Science Center, Poland (NCN),
UMO-2013/09/D/ST1/03681.}

\begin{abstract}
    We study solutions and supersolutions of homogeneous and nonhomogeneous $\A$-harmonic equations with nonstandard growth in $\mathbb{R}^n$. Various Liouville-type theorems and nonexistence results are proved. The discussion is illustrated by a number of examples.
\end{abstract}
\bigskip

\medskip


\long\def\symbolfootnote[#1]#2{\begingroup%
\def\thefootnote{\fnsymbol{footnote}}\footnote[#1]{#2}\endgroup}

\section{Introduction}

The purpose of this paper is to investigate Liouville theorems in the setting of quasilinear equations of the following type:
\begin{equation*}
    -\div \A(x, u, \nabla u)=\B (x, u, \nabla u),
\end{equation*}
where $\A$ and $\B$ are differential operators with nonstandard growth. The prototype of such equations is the so-called $\p$-Laplacian
\begin{equation*}
  -\div \left(|\nabla u|^{\p-2} \nabla u\right)=0.
\end{equation*}
Variable exponent equations grow largely from the studies in calculus of variations and applied sciences, such as
fluid dynamics, see Diening--R\r u\v zi\v cka~\cite{dr}, electro-rheological fluids, see e.g. Acerbi--Mingione~\cite{am02}, image processing, see Chen--Levine--Rao~\cite{clr} and models for thermistor, see Zhikov~\cite{z97}; see also Harjulehto--H\"ast\"o--L\^e--Nuortio~\cite{hhn} for a recent survey and further references and R\r u\v zi\v cka~\cite{Ru00} for more details on the role of nonhomogeneous $\p$-harmonic equations in applications.

In spite of the symbolic similarity to the constant exponent $p$-harmonic equations, a number of phenomena occur
when the exponent is a function, for example the minimum of the
$\p$-Dirichlet energy may not exist even in the one-dimensional
case for a smooth function $p$. Moreover, smooth functions need not to be
dense in the corresponding variable exponent Sobolev spaces, on the contrary to the constant exponent case.
In the context of the main theme of our work, Example~\ref{przyklad1} below shows that even in the relatively simple case of $\p$-harmonic functions on the real line, the Liouville theorem need not hold in general.

In Preliminaries we recall and introduce some basic notions and definitions. Then, in Section~3 we prove various Liouville type theorems for homogeneous $\A$-harmonic equations. Our proofs rely largely on Caccioppoli-type inequalities and on the choice of appropriate test functions. However, the studies in variable exponent setting lead to a number of difficulties, such as lack of simple relations between the norm and modular function, non-homogeneous Poincar\'e inequalities for modular functions and non-homogeneous Harnack inequalities for equations of $\p$-harmonic type. In a consequence, one needs to require stronger assumptions on variable exponents comparing to the constant exponent case and appeal to some trick-arguments, for instance the hole-filling technique. Nevertheless, results of Theorems~\ref{thm:bdd-grad}, \ref{thm:thm3} and Corollary~\ref{cor:thm3} partially generalize Theorem 2 and Corollary 2 in Mitidieri--Pokhozaev~\cite{MP} obtained for $\A$-harmonic operators with constant growth. One of  advantages of our approach is that we need growth conditions on exponent functions only on annuli. In Section~3 we also comment on connections between Liouville theorems and Phragm\'en--Lindel\"of theorems and state two open problems.

 Section~4 is devoted to studying the non-homogeneous $(\A,\B)$-harmonic equations. In Theorem~\ref{thm:bdd-grad-A-B} we show the Liouville-type result for general type of operators $\B$. In subsequent results we study two cases of $\B$: depending only on supersolution $u$ and Riccati-type inequalities with $\B$ depending on $|\nabla u|$ only. Moreover, we show some nonexistence results depending on monotonicity and convexity properties of operators $\B$.

Our results are natural extensions of Liouville theorems for $\A$-harmonic equations with constant growth, see e.g. Caristi--Mitidieri~\cite{CM} and Mitidieri--Pokhozaev~\cite{MP}. For some further results in the topic of Liouville theorems we also refer to D'Ambrosio~\cite{A}, D'Ambrosio--Mitidieri~\cite{AM1, AM2, AM3}, Filippucci~\cite{Fil} and Serrin~\cite{Serrin}.

We would like to emphasize that, according to our best knowledge, the Liouville-type theorems in the setting of equations with nonstandard growth have not yet been studied systematically in the literature. In fact, it appears that there exists only one paper about the Liouville theorem in the context of variable exponent, see Wang~\cite{Wang};
see also Pucci--Zhang~\cite{PZ} and Dinu~\cite{Dinu} for related topics. We hope that our results will attract wider audience and lead to deeper studies of Liouville type theorems and nonexistence results for PDEs with nonstandard growth.

\section{Preliminaries}
We let denote $B(x, r)=\{y \in \R^{n} : |x - y|<r\}$ a
ball centered at a point $x$ with radius $r>0$. In what follows we will often write
$B(x, r)=B_r$. Let $dx$ denote the $n$-dimensional Lebesgue measure on $\R^{n}$ and let $\omega(n)$ stand for the measure of a unit ball in $\R^n$. If $\Om \subset
\R^n $ is open and $1 \leq q < \infty$, then by $W^{1 ,q} (\Om)$,
$W^{1 ,q}_0 (\Om)$ we denote the standard Sobolev space and the Sobolev space of functions with zero boundary values, respectively. By $f_A$ we denote the integral average of a function $f$ over a set $A$, that is
\[
 f_A:=\vint_{A} f dx=\frac{1}{|A|}\int_{A} f dx.
\]
We now turn to a brief presentation of variable exponent theory. For background on variable exponent function spaces we refer to the monograph by Diening--Harjulehto--H\"ast\"o--R\r u\v zi\v cka~\cite{DHHR}.

A measurable function $p\colon \Omega\to [1,\infty]$ is called a
\emph{variable exponent}, and we denote
\[
p^+_A:=\esssup_{x\in A} p(x),\quad p^-_A:=\essinf_{x\in A} p(x),
\quad p^+:=p_\Omega^+ \quad\text{and}\quad p^-:=p_\Omega^-
\]
for $A\subset\Omega$.  If $A=\Om$ or if the underlying domain is fixed, we will often skip the index and set $p_A=p_\Om=p$.

The function $\alpha$ defined in a bounded domain $\Omega$ is said to be
\textit{$\log$-H\"older continuous} if there is constant
$L>0$ such that
\[
|\alpha(x)-\alpha(y)|\leq \frac{L}{\log(e+1/|x-y|)}
\]
for all $x,y\in \Omega$. We denote $p\in \PPln(\Omega)$ if
$1/p$ is $\log$-H\"older continuous; the smallest constant for which $\frac1p$ is
$\log$-H\"older continuous is denoted by $c_{\log}(p)$.

 In this paper we study only bounded log-H\"older continuous or bounded Lipschitz continuous variable exponents,
 i.e. we assume that $p\in \PPln(\Omega)$ and
\[
 1<p^-\leq p(x)\leq p^+<\infty\qquad \hbox{for almost every } x\in \Om.
\]
 For bounded variable exponents it holds that $1/p$ is $\log$-H\"older continuous if and only if $p$ is such, cf.
 \cite[Remark 4.1.5]{DHHR}.

 Both types of exponents can be extended to the whole $\R^n$ with their constants unchanged, see \cite[Proposition 4.1.7]{DHHR} and McShane-type extension result in Heinonen~\cite[Theorem~6.2]{hei01}, respectively. Therefore, without loss of generality we assume below that variable exponents are defined in the whole $\R^n$.

We define a \emph{(semi)modular} on the set of measurable functions
by setting
\[
\varrho_{L^{p(\cdot)}(\Omega)}(u) :=\int_{\Omega} |u(x)|^{p(x)}\,dx;
\]
here we use the convention $t^\infty = \infty \chi_{(1,\infty]}(t)$
in order to get a left-continuous modular, see
\cite[Chapter~2]{DHHR} for details. The \emph{variable
exponent Lebesgue space $L^{p(\cdot)}(\Omega)$} consists of all
measurable functions $u\colon \Omega\to \R$ for which the modular
$\varrho_{L^{p(\cdot)}(\Omega)}(u/\mu)$ is finite for some $\mu >
0$. The Luxemburg norm on this space is defined as
\[
\|u\|_{L^{p(\cdot)}(\Omega)}:= \inf\Big\{\mu > 0\,\colon\,
\varrho_{L^{p(\cdot)}(\Omega)}\big(\tfrac{u}\mu\big)\leq 1\Big\}.
\]
Equipped with this norm, $L^{p(\cdot)}(\Omega)$ is a Banach space. The variable exponent Lebesgue space is a special case of an Orlicz-Musielak space. For a constant function $p$, it coincides
with the standard Lebesgue space. Often it is assumed that
$p$ is bounded, since this condition is known to imply many desirable
features for $L^{p(\cdot)}(\Omega)$.

If $E$ is a measurable set of finite measure, and $p$ and $q$
are variable exponents satisfying $q\leq p$, then $L^{p(\cdot)}(E)$
embeds continuously into $L^{q(\cdot)}(E)$. In particular,
every function $u\in L^{p(\cdot)}(\Omega)$ also belongs to
$L^{p_\Omega^-}(\Omega)$. The variable exponent H\"older inequality
takes the form
\begin{equation*}
\int_\Omega f g \,dx \le 2 \, \|f\|_{L^{p(\cdot)}(\Omega)} \|g\|_{L^{p'(\cdot)}(\Omega)},\label{ineq:Holder}
\end{equation*}
where $p'$ is the point-wise \textit{conjugate exponent}, $1/p(x)
+1/p'(x)\equiv 1$.
\bigskip

The \emph{variable exponent Sobolev space $W^{1,p(\cdot)}(\Omega)$}
consists of functions $u\in L^{p(\cdot)}(\Omega)$ whose
distributional gradient $\nabla u$ belongs to
$L^{p(\cdot)}(\Omega)$. The variable exponent Sobolev space
$W^{1,p(\cdot)}(\Omega)$ is a Banach space with the norm
\begin{displaymath}
\|u\|_{L^{p(\cdot)}(\Omega)}+\|\nabla u\|_{L^{p(\cdot)}(\Omega)}.
\end{displaymath}
In general, smooth functions are not dense in the variable exponent
Sobolev space, see Zhikov~\cite{Zh06} but the log-H\"older condition suffices
to guarantee that they are, see Diening--Harjulehto--H\"ast\"o--R\r u\v zi\v cka~\cite[Section~8.1]{DHHR}.
In this case, we define \emph{the Sobolev space with zero
boundary values}, $W_0^{1,\p}(\Omega)$, as the closure of $C_0^\infty(\Omega)$
in $W^{1,\p}(\Omega)$.

Below we will also frequently use the $(1,q)$-Poincar\'e inequality for a constant $1<q<\infty$: if $v\in W^{1,q}_{loc}(\Om)$, then
\begin{equation}\label{poincare-ineq}
\int_{B_r}|v-v_{B_r}|^q\leq Cr^q \int_{B_r}|\nabla v|^q,
\end{equation}
where $v_{B_r}$ denotes the mean value of $v$ over the ball $B_r$ and $C$ depends on $n$ and $q$.

The main focus of our studies will be various types of $\A$- and $(\A, \B)$-harmonic equations with nonstandard growth.
The prototypical $\A$-harmonic equation with nonstandard growth is the so-called $\p$-harmonic equation.

\begin{defn}\label{px-defn}

A function $u \in W^{1, \p}_{loc}(\Om)$ is a \/\textup{(}sub\/\textup{)}solution if
\begin{equation} \label{eq-sub-sol}
       \int_\Om |\nabla u|^{\px-2} \langle \nabla u, \nabla \phi\ \rangle dx (\le)= 0
\end{equation}
for all\/ \textup{(}nonnegative\/\textup{)} $\phi \in C_0^\infty(\Om)$.
\end{defn}
Similarly, we say that $u$ is a \emph{supersolution ($\p$-supersolution)} if $-u$ is a subsolution. A
function which is both a subsolution and a supersolution is called a (weak) solution to the $\p$-harmonic equation.
A continuous weak solution is called a \emph{$\p$-harmonic function}.

Recall that a function $f: \Om\times\R\times\Rn\to \R$ is called \emph{a Carath\'eodory function} if it is measurable of the first argument and continuous in the latter two arguments.
\begin{defn}\label{A-function}
  Let $p:\Om\to [1,\infty]$ be a variable exponent. The Carath\'eodory function $\A:\Om\times\R\times\Rn\to \R$ is called \emph{of nonstandard growth} or \emph{$\p$-growth} if there exist functions $a, b$ such that the following conditions hold for all $(x, t, \xi) \in \Om\times\R\times\Rn\to \R$:
 \begin{itemize}
 \item[(1)] $\langle \A(x, t, \xi), \xi \rangle \geq a(x) |\xi|^{\px}$, where $a:\Om\to\R$ is a measurable function bounded from below by $a_{-}>0$.
 \item[(2)] $|\A(x, t, \xi)| \leq b(x) |\xi|^{\px-1}$, where $b:\Om\to\R$ is a measurable function such that $a\leq b$ in $\Om$ and bounded from above by $b^{+}<\infty$.
 \end{itemize}
 \end{defn}
The most important example of function $\A$ with nonstandard growth is $\A=|\xi|^{\px-2}\xi$ and the associated  $\p$-harmonic differential operator and equation. If $p=const$, then we retrieve the well-known $\A$-harmonic operator, see e.g. Heinonen--Kilpel\"ainen--Martio~\cite[Chapter 3]{hkm}. Note however, that here we do not require monotonicity of the operator $\A$. Moreover, the standard for constant exponent homogeneity assumption on $\A$ would not make too much sense in the exponent setting.

Let us also mention that one can study more general growth conditions for $\A$ functions, e.g. require
$\langle \A(x, t, \xi), \xi \rangle \geq a(x) (\kappa+|\xi|^2)^{\frac{\px-2}{2}}|\xi|^2$ for some $\kappa>0$ and an analogous assumption for the upper bound, see for instance Fan~\cite{Fan} and Harjulehto--H\"ast\"o--L\^e--Nuortio~\cite{hhn}. However, such assumptions lead to unhandy forms of the Caccioppoli-type estimates and, in turn, to tedious technical discussions for Liouville theorems. Therefore, for the sake of simplicity of the presentation we focus on the case with $\kappa=0$.

\begin{defn}\label{A-harm}
 Let $\A, \B: \Om\times\R\times\Rn\to \R$ be Carath\'eodory functions. Suppose that $\A$ is of $\p$-growth and $|\B(x, t, \xi)|\leq c(x)(1+|\xi|^{\px-1})$ for some bounded from above measurable functions $c:\Om\to\R$ such that $\sup_{\Om} c=c^{+}<\infty$.
 A function $u \in W^{1, \p}_{loc}(\Om)$ is called an $(\A,\B)$-\/\textup{(}sub\/\textup{)}solution if
\begin{equation} \label{eq-sub-sol}
    \int_\Om \langle \A(x, u, \nabla u), \nabla \phi\rangle \, dx \,(\leq)=\, \int_\Om \B(x, u, \nabla u) \phi \, dx
\end{equation}
for all\/ \textup{(}nonnegative\/\textup{)} $\phi \in C_0^\infty(\Om)$.
\end{defn}
In the analogous way we define \emph{$(\A, \B)$-supersolutions}, called for the brevity \emph{supersolutions}. A function which is both a subsolution and a supersolution is called a (weak) solution to the $(\A, \B)$-harmonic equation. A continuous weak solution is called an \emph{$(\A, \B)$--harmonic function}. If $\B \equiv 0$ we call \eqref{eq-sub-sol} an \emph{$\A$-harmonic equation} and study \emph{$\A$-supersolutions} and \emph{$\A$-harmonic functions}.

\section{Liouville theorems for $\A$-harmonic functions}

In this section we study the Liouville-type theorems for $\A$-harmonic equations with the right-hand side zero, i.e. for equations as in Definition~\ref{A-harm} with $\B\equiv 0$. Let us discuss our results in the context of studies previously done for Liouville theorems in the setting of elliptic equations. In the case of $p=const$ operators $\A$ as in Definition~\ref{A-function} have been investigated by e.g. Mitidieri--Pokhozaev, see condition (11) in \cite{MP}.
Results of our Theorems~\ref{thm:bdd-grad}, \ref{thm:thm3} and Corollary~\ref{cor:thm3} partially generalize Theorem 2
and Corollary 2 in \cite{MP}. There $u$ is assumed to be nonnegative, while we assume that $u$ is bounded. On the other hand, we impose the growth conditions for the exponent $p$ only on annuli. For the similar studies in the constant exponent case we refer to e.g. Theorems 1 and 2 in Serrin~\cite{Serrin} and Theorem 3.2 in D'Ambrosio~\cite{A}. One can also wider a context of our discussion by noticing that the Liouville property for $\A$-harmonic operators of the $p$-harmonic type is equivalent to the $p$-parobolicity of $\Rn$, cf. Theorem~5.4 in Holopainen--Pankka~\cite{h-p} and discussion therein in the setting of manifolds.

Let us start our discussion with an example showing that already in the relatively simple case of $\p$-harmonic functions on the real line $\R$ the Liouville theorem may not hold.

\begin{ex}\label{przyklad1}
Let $p:\mathbb{R} \rightarrow (1,\infty)$ be a Lipschitz variable exponent defined as follows: $p(x)=1+\frac{1}{1+|x|}$. Define
\begin{equation*}
u(t) := \left\{ \begin{array}{l l}
               \frac{1}{e}(1-e^{-t})  & t \geq 0  \\
                -\frac{1}{e}(1-e^{t}) & t < 0.
                 \end{array} \right.
\end{equation*}
Then $u$ is a nonconstant bounded $\p$-harmonic function in $\mathbb{R}$. Indeed, $u\in C^1(\R)$ and so, in particular, $u\in W^{1,1}_{loc}(\R)$. The $\p$-harmonic equation in dimension $n=1$ reads: $\Delta_{\p}(u)(t):=(|u'(t)|^{p(t)-2}u'(t))'=0$ for all $t\in I$ for any open interval $I\subset \R$. One can check by direct computations that the above exponential function satisfies the $\p$-harmonic equation.

Observe that $p(x)>1$ for all $x\in \R$, but $p^{-}=\inf_{x\in \R}p(x)=1$. This is a reflection of the well-known phenomenon in the theory of variable exponent equations that several theorems may break if we allow $p^{-}=1$, cf. discussion in Section II.3 of Harjulehto--H\"ast\"o--L\^e--Nuortio~\cite{hhn}. However, see Corollary~\ref{cor:thm3} below.
\end{ex}

The following auxiliary result is well-known. For the convenience of readers and sake of completeness we recall its proof.

\begin{lem}[The Caccioppoli inequality for $\A$-harmonic equations]\label{lem:Cac-ineq}
 Let $u$ be a subsolution to an $\A$-harmonic equation in a domain $\Om\subset \Rn$ for $\A$-functions
 as in Definition~\ref{A-harm}. Then for $0<r<R$ and $c\leq \inf_{\Om}u$ we have that
 \begin{equation}\label{Cac-ineq}
 \int_{B_r} |\nabla u|^{\px}\leq
  C(p^+, a_{-})\int_{B_R\setminus B_r}\left(\frac{|u-c|}{R-r}\right)^{\px}.
 \end{equation}
 Furthermore, if $u$ is an $\A$-harmonic solution in $\Om$, then $c$ can be any real number.
\end{lem}

\begin{proof}
Let $0<r<R$ and let $\eta\in C_{0}^{\infty}(\Om)$ be such that $\eta|_{B_r}=1$ with $\rm{supp}\,\eta \subset B_R$ and
$|\nabla \eta|\leq \frac{C}{R-r}$. Define $\phi=(u-c)\eta^{p^+}\geq 0$ for $c\leq \inf_{\Om}u$. Then
\begin{equation*}
 \nabla \phi=\eta^{p^+}\nabla u +p^{+}\eta^{p^+-1}(u-c)\nabla \eta.
\end{equation*}
We use $\phi$ as a test function in \eqref{eq-sub-sol} for a subsolution $u$ and, upon using the lower and upper bounds of $\A$, cf. Definition~\ref{A-function}, we get the following inequality:
\begin{align}
 a_{-}\int_{B_R} |\nabla u|^{\px} \eta^{p^+}&
 \leq p^+b^{+}\int_{B_R} |u-c||\nabla u|^{\px-1} |\nabla \eta| \eta^{p^+-1}. \label{ineq:Cac-AB}
\end{align}
Upon employing the Young inequality for some $0<\epsilon<1$, we get the following estimate of the term on the right-hand side of \eqref{ineq:Cac-AB}:
\begin{align*}
p^+b^{+}\int_{B_R} |u-c||\nabla u|^{\px-1} |\nabla \eta| \eta^{p^+-1}
 \leq \epsilon p^+ b^{+} \int_{B_R}|\nabla u|^{p(x)}\eta^{p^+}+\frac{b^{+}}{\epsilon^{p^+}}\int_{B_R}|u-c|^{p(x)}|\nabla \eta|^{p(x)}.
\end{align*}
Take $\epsilon=\frac{a_{-}}{2p^+b^{+}}$ and include the $|\nabla u|$-term on the right-hand side into the left-hand side of  \eqref{ineq:Cac-AB}. Finally, we use the growth assumption on $|\nabla \eta|$ to get the assertion of the lemma.
\end{proof}

We begin the presentation of Liouville-type theorems with the following result. The most relevant example of differential operators satisfying assumptions of our theorems is provided by $\A(x,u,\nabla u)=|\nabla u(x)|^{\px-2}\nabla u(x)$.

\begin{theorem}\label{thm:bdd-grad}
 Suppose that there exists $\delta >0$ and a divergence increasing sequence $\{R_k\}_{k=1}^{\infty}$ such that $$p^{-}_{ B_{2R_k}\setminus B_{R_k}}\geq n+\delta.$$
 If $u$ is a bounded $\A$-harmonic function in $\Rn$, then $u\equiv \hbox{const}$.
\end{theorem}
\begin{proof}
 Suppose that $|u|\leq M$ in $\Rn$. Using the Caccioppoli inequality~\eqref{Cac-ineq} with $R=2R_k$ and $r=R_k$ we get for $R_k>1$:
\begin{align*}
 \int_{B_{R_k}}|\nabla u|^{p(x)} &\leq C(p^+, a_{-})
 \int_{B_{2R_k}\setminus B_{R_k}}\left(\frac{|u|}{R_k}\right)^{\px}\\
 &\leq C(p^+, a_{-}) (M+1)^{p^+}\int_{B_{2R_k}\setminus B_{R_k}}\frac{1}{R_k^{\px}} \\
 &\leq C(p^+, a_{-}) (M+1)^{p^+}\omega(n)\frac{R_k^n}{R_k^{p^{-}_{B_{2R_k}\setminus B_{R_k}}}}.
\end{align*}
 Let $k\rightarrow \infty$ in the above inequality, then we conclude:
\begin{equation*}
 \int_{\mathbb{R}^n}|\nabla u|^{p(x)}=0.
\end{equation*}
Hence $\nabla u=0$ a.e. and $u=const$.
\end{proof}

In the next theorem we again impose the growth assumption on the variable exponent on annuli, but we do not require $p^{-}>n$. Instead, we need the global integrability for the gradient of $\A$-harmonic functions.

\begin{theorem}\label{thm:thm3}
 Let $n \geq 2$. Suppose that there exists a divergence increasing sequence $\{R_k\}_{k=1}^{\infty}$ such that
 for some constant $q$ with $1\leq q \leq p^+<\infty$ it holds
 \begin{equation*}
 p|_{B_{2R_k}\setminus B_{R_k}}=q \quad \hbox{ for } k=1,2,\ldots.
\end{equation*}
If $u$ is an $\A$-harmonic function in $\mathbb{R}^n$ and satisfies $\int_{\mathbb{R}^n}|\nabla u|^{p(x)}<\infty$, then $u\equiv \hbox{const}$.
\end{theorem}

\begin{proof}[Proof of Theorem~\ref{thm:thm3}]
First, we use the Caccioppoli inequality~\eqref{Cac-ineq} with $c=u_{B_{2R_k}\setminus B_{R_k}}:=\frac{1}{|B_{2R_k}\setminus B_{R_k}|}\int_{B_{2R_k}\setminus B_{R_k}}u$ and obtain the following:
\begin{equation*}
 \int_{B_{R_k}}|\nabla u|^{p(x)}\leq C(p^+, a_{-}) \int_{B_{2R_k}\setminus B_{R_k}}
 \left(\frac{|u-u_{B_{2R_k}\setminus B_{R_k}}|}{R_k}\right)^{q}.
\end{equation*}
Then the $q$-Poincar\'e inequality \eqref{poincare-ineq} used on sets $B_{2R_k}\setminus B_{R_k}$ gives us that
 \begin{equation*}
 \int_{B_{R_k}}|\nabla u|^{p(x)}\leq C(n, p^+, a_{-}, q) \int_{B_{2R_k}\setminus B_{R_k}} |\nabla u|^{q}=C(n, p^+, a_{-}, q) \int_{B_{2R_k}\setminus B_{R_k}} |\nabla u|^{p(x)}.
\end{equation*}
 Hence, by using the so-called filling holes techniques we obtain
 \begin{equation*}
 \int_{B_{R_k}}|\nabla u|^{p(x)}\leq \frac{C(n, p^+, a_{-}, q)}{1+C(n, p^+, a_{-}, q)}\int_{B_{2R_k}}|\nabla u|^{\px}.
\end{equation*}
Let $k\to \infty$, then we get the contradiction $\int_{\R^n}|\nabla u|^{p(x)}<\int_{\R^n}|\nabla u|^{p(x)}$ unless
$\int_{\R^n}|\nabla u|^{p(x)}=0$. Then $u$ is constant and the proof is completed.
\end{proof}

\begin{opprb}
 In Theorem~\ref{thm:thm3} we allow a variable exponent $p$ to take only one fixed value $q$ on annuli. It would be interesting to study the case when one allows possibly different values $q_k\in [1, p^{+}]$ on each annuli. This however leads to a problem whether the sequence of constants $\{C(n, p^+, a_{-}, q_k)\}_{k=1}^{\infty}$, arising in the proof of Theorem~\ref{thm:thm3} from the $q$-Poincar\'e inequalities, is uniformly bounded. According to our best knowledge such a bound is not known in the literature, largely due to the fact that annuli are not convex sets.
\end{opprb}

If in the previous result we assume additionally that $q\geq n$, then the global integrability of the gradient follows. Hence, Theorem~\ref{thm:thm3} implies the following result.

\begin{cor}\label{cor:thm3}
Under the assumptions of Theorem~\ref{thm:thm3} let additionally variable exponent $\p$ satisfy
$p|_{B_{2R_k}\setminus B_{R_k}}=q \geq n$ for all $k=1,2,\ldots$. It holds that if $u$ is a bounded $\A$-harmonic function in $\mathbb{R}^n$, then $u\equiv \hbox{const}$.
\end{cor}
\begin{proof}
 We show that our assumptions imply that $\int_{\Rn}|\nabla u|^{\px}<\infty$ whenever $u$ is bounded. Indeed, let $|u|\leq M$ in $\Rn$. We proceed as in the proof of Theorem~\ref{thm:bdd-grad} and by using the Caccioppoli inequality~\eqref{Cac-ineq} together with the bounds for $p$ we obtain for $R_k>1$ the following estimate:
\begin{align*}
 \int_{B_{R_k}}|\nabla u|^{p(x)}&\leq C(p^+, a_{-}) \int_{B_{2R_k}\setminus B_{R_k}}\left(\frac{|u|}{R_k}\right)^{\px} \\
 &\leq C(p^+, a_{-})(M+1)^{p^+} \int_{B_{2R_k}\setminus B_{R_k}}\frac{1}{R_k^{q}}\\
 &\leq C(p^+, a_{-})(M+1)^{p^+} (2^n-1) \omega(n) \frac{1}{R_k^{n-q}}\\
 &\leq C(n, p^+, a_{-})(M+1)^{p^+}\omega(n).
\end{align*}
 Since the upper bound does not depend on $k$ and holds uniformly for all $R_k>1$, we get that $\int_{\Rn}|\nabla u|^{\px}<\infty$. Then Theorem~\ref{thm:thm3} implies that $u\equiv const$.
\end{proof}
In the next theorem we study the case of $p^-<n$. Then, one needs stronger assumptions on the growth of $\A$-harmonic functions in order for the Liouville principle to hold.

\begin{theorem}\label{thm:p<n}
Let $n\geq 2$ and $p$ be a bounded variable exponent on $\Rn$ such that $p^-<n$.
Suppose that $\alpha \in \R$ is such that $\alpha>\frac{n}{p^-}-1$. Suppose that
there exists a divergence increasing sequence $\{R_k\}_{k=1}^{\infty}$ such that $u_k:=u|_{B_{2R_k}\setminus
B_{R_k}}$ satisfies $u_k\leq \frac{c}{R_k^{\alpha}}$ for $k=1,2,
\ldots$.

If $u$ is a bounded $\A$-harmonic function in $\mathbb{R}^n$, then $u\equiv \hbox{const}$.
\end{theorem}
\begin{proof}
We proceed similarly to the proof of Theorem~\ref{thm:bdd-grad} and
by applying the Caccioppoli inequality~\eqref{Cac-ineq} with $c=0$
we get for $R_k>\max\{1, \sup_{\R^n}u\}$
\begin{align*}
 \int_{B_{R_k}}|\nabla u|^{p(x)}&\leq C(p^+, a_{-}) \int_{B_{2R_k}\setminus B_{R_k}}\left(\frac{|u|}{R_k}\right)^{\px}\\
&\leq C(p^+, a_{-})\omega (n)\left(\frac{\sup|u_k|}{R_k}\right)^{p^{-}_{2R_k\setminus
 R_k}} R_k^n\\
&\leq \frac{c}{R_k^{p^-\alpha+p^--n}}<\infty,
 \end{align*}
as $p^{-}_{B_{2R_k}\setminus  B_{R_k}}\geq p^{-}_{B_{2R_k}}\geq p^-$. Let $k\rightarrow \infty$ in the above inequality, then we
conclude:
\begin{equation*}
 \int_{\mathbb{R}^n}|\nabla u|^{p(x)}=0.
\end{equation*}
Hence $\nabla u=0$ a.e. and $u=const$.
\end{proof}

Lemma~\ref{lem:Cac-ineq} implies also a variant of the famous Phragm\'en--Lindel\"of theorem. Theorems of such type have been intensively studied both in the setting of linear PDEs and for nonlinear ones, including the $p$-harmonic case, see Lindqvist~\cite{lin}; see also Adamowicz~\cite{ad1} for similar studies in the setting of variable exponent PDEs.
\begin{cor}\label{Ph-Lind}
 Let $u$ be an $\A$-harmonic in $\mathbb{R}^n$ for $\A$ with nonstandard growth as in Definition~\ref{A-function}. Suppose that there exist a constant $C>0$ and an increasing unbounded sequence of radii $\{R_k\}_{k=1}^{\infty}$ such that $u$ satisfies the following growth condition for all $x\in B_{2R_k}\setminus B_{R_k}$ starting from some $R_k>1$:
 \begin{equation*}
  |u(x)|\leq C|x|^{\alpha}\quad \hbox {for }\quad \alpha<\frac{p^-}{p^+}-\frac{n}{p^+}.
 \end{equation*}
Then $u$ is constant.
\end{cor}
\begin{rem}
 If $p^-=p^+=p=const$, then the growth assumption holds for $\alpha<1-\frac{n}{p}$. In the special case of harmonic functions (i.e. $p=n=2$) or $n$-harmonic functions (i.e. $p=n$), this gives us that $\alpha$ must satisfy $\alpha<0$.
 Similarly, for $p$-harmonic functions we get for $p\geq \frac{n}{2}$ that $\alpha\leq -1$.

 Suppose now that the entire $p$-harmonic function $u\not\equiv const$ and $\alpha<1-\frac{n}{p}$. Then the corollary implies that $\liminf_{R\to \infty}\frac{|u(x)|}{|x|^{\alpha}}> C$. In particular, by taking $\alpha=-1$ we retrieve the growth condition from the Phragm\'en--Lindel\"of theorem.
\end{rem}
\begin{proof}[Proof of Corollary~\ref{Ph-Lind}]
 We apply Lemma~\ref{lem:Cac-ineq} with $c=0$, $R=2R_k$ and $r=R_k$ together with the growth assumption on $u$ and get the following inequality for $R_k>1$:
 \begin{equation*}
 \int_{B_{R_k}} |\nabla u|^{\px}\leq
  C(p^+, a_{-})\int_{B_{2R_k}\setminus B_{R_k}}\frac{|x|^{\alpha \px}}{R_k^{\px}}\leq C(p^+, a_{-})\omega(n)\frac{1}{R_k^{p_{-}}}R_k^{n+\alpha p_{+}}.
 \end{equation*}
 From this, the assertion of the corollary follows immediately, since letting $R_k\to \infty$ we obtain that
 $\int_{\Rn} |\nabla u|^{\px}=0$ and so $u$ is constant.
\end{proof}
\begin{opprb}
 We have seen in Example \ref{przyklad1} that the Liouville theorem may not hold if $p^{-}=1$. Furthermore, in Theorems \ref{thm:bdd-grad} and \ref{thm:thm3} we need to restrict the range of $\p$. On the other hand, if $p=const$, then the Liouville theorem holds for $p$-harmonic functions for any $1<p<\infty$. Hence, it is natural to state the following problem. If we only assume that $p\in \PPln(\Rn)$ and $p^- >1$, determine whether the Liouville theorem holds for $\p$-harmonic functions. And more general, under what general conditions on the exponent $\p$ does the Liouville theorem hold for $\p$-harmonic functions?
\end{opprb}

\section{Liouville theorems for $(\A,\B)$-harmonic functions}

The goal of this section is to study the Liouville principle for nonhomogeneous equations. First, we present a variant of Caccioppoli inequality for positive solutions of $(\A, \B)$-harmonic equations.
\begin{lem}\label{lem:Cac-ineq-positive}
  Let $u>0$ be a bounded $(\A, \B)$-harmonic function in a domain $\Om\subset \Rn$ for $\A$- and $\B$-functions as in Definition~\ref{A-harm}.  Then for $0<r<R$ and $\gamma <0$ we have that
 \begin{equation}\label{Cac-ineq-positive}
 \int_{B_r} |\nabla u|^{\px}u^{\gamma-1}\leq
  C\int_{B_R\setminus B_r}\left(\frac{|u|}{R-r}\right)^{\px+\gamma-1}+
  \frac{2}{a_{-}|\gamma|}\int_{B_R}\B(x, u, \nabla u) u^{\gamma}\eta^{p^{+}_{B_{R}}}.
 \end{equation}
 where constant $C=C(p^+, a_{-}, b^{+}, |\gamma|)$.
\end{lem}
\begin{rem}
 If in the above lemma we consider $\gamma>0$, then the result holds also for $(\A, \B)$-harmonic subsolutions.
\end{rem}
\begin{proof}[Proof of Lemma~\ref{lem:Cac-ineq-positive}]
Let $0<r<R$ and let $\eta\in C_{0}^{\infty}(\Om)$ be such that $\eta|_{B_r}=1$ with $\rm{supp}\,\eta \subset B_R$ and
$|\nabla \eta|\leq \frac{C}{R-r}$. Define $\phi=u^{\gamma}\eta^{p^+}$. Since $\gamma<0$, the issue of regularity of $u^{\gamma}$ and $\phi$ requires some extra discussion. First, notice that as $u$ is continuous and $\eta$ has a compact support, then $\phi\in L^{\infty}({\rm supp}\,\phi)$. Furthermore, by using the embedding properties of the Sobolev spaces we get that $u\in W^{1,p^{-}}_{loc}(\Om)$.
Then we compute
\begin{equation*}\label{eq:Cac-ineq-pos}
 \nabla \phi=\gamma u^{\gamma-1}\nabla u\eta^{p^+} +p^{+}u^{\gamma}\eta^{p^+-1}\nabla \eta.
\end{equation*}
By using the standard uniform approximation property for continuous functions in constant Sobolev spaces, together with the properties of test function $\eta$, we conclude that $\phi\in W^{1, p^{-}}_{0}(\Om)$. Finally, the following pointwise estimate gives us that $|\nabla \phi|\in L^{\p}_{loc}(\Om)$ and thus $\phi\in W^{1, \p}_{0}(\Om)$:
\[
 |\nabla \phi|^{\px}\leq 2^{p^{+}}\left(\gamma^{p^{+}} (1+u^{\gamma-1})^{p^{+}}|\nabla u|^{\px}(\eta^{p^+})^{p^{+}} +(p^{+}\eta^{p^+-1})^{p^{+}}(1+u^{\gamma})^{p^{+}}|\nabla \eta|^{\px}\right).
\]
We use $\phi$ as a test function in \eqref{eq-sub-sol} for a solution $u$ and, upon using the lower and upper bounds of $\A$, cf. Definition~\ref{A-function}, we get the following inequality:
\begin{align}
 a_{-}|\gamma|\int_{B_R} |\nabla u|^{\px}u^{\gamma-1} \eta^{p^+}&
 \leq p^+\int_{B_R} |\A(x, u, \nabla u)|u^{\gamma}|\nabla \eta| \eta^{p^+-1}+\int_{B_R} \B(x, u, \nabla u)u^{\gamma} \eta^{p^+} \nonumber \\
 &\leq p^+b^{+}\int_{B_R} |\nabla u|^{\px-1}u^{\gamma} |\nabla \eta| \eta^{p^+-1}+\int_{B_R} \B(x, u, \nabla u)u^{\gamma} \eta^{p^+}. \label{ineq:Cac-AB2}
\end{align}
Upon employing the Young inequality for some $0<\epsilon<1$, we get the following estimate of the first term on the right-hand side of \eqref{ineq:Cac-AB2}:
\begin{align*}
p^+b^{+}\!\!\int_{B_R}\!\!\left(|\nabla u|^{\px-1} u^{(\gamma-1)\frac{\px-1}{\px}} \eta^{p^+-1}\right) \left(
u^{1+\frac{\gamma-1}{\px}}|\nabla \eta| \right)
 \leq \epsilon p^+ b^{+}\!\!\int_{B_R}\!\!|\nabla u|^{p(x)}u^{\gamma-1}\eta^{p^+}+\frac{b^{+}}{\epsilon^{p^+}}\!\int_{B_R}\!\!u^{\px+\gamma-1}|\nabla \eta|^{p(x)}.
\end{align*}
Take $\epsilon=\frac{a_{-}|\gamma|}{2p^+b^+}$ and include the $|\nabla u|$-term on the right-hand side into the left-hand side of  \eqref{ineq:Cac-AB2}. Finally, we use the growth assumption on $|\nabla \eta|$ to conclude the assertion of the lemma.
\end{proof}

In the next theorem we study an example of the general condition imposed on the right-hand side of an $(\A, \B)$-harmonic equation implying the Liouville theorem.

\begin{theorem}\label{thm:bdd-grad-A-B}
 Suppose that there exist $\delta>0$, $\gamma<0$ and a divergence increasing sequence $\{R_k\}_{k=1}^{\infty}$ such that $$p^{-}_{ B_{2R_k}\setminus B_{R_k}}\geq n+\delta +1-\gamma.$$
Additionally, let us assume that $\B(\cdot, u(\cdot), \nabla u(\cdot)) u(\cdot)^{\gamma} \in L^1(\mathbb{R}^n)$ as a function of $x\in\Rn$ and satisfies the following integral condition:
 \begin{equation*}\label{Cac-B-growth}
 \lim_{R\to \infty} \int_{B_{R}} \B(x, u(x), \nabla u(x)) u(x)^{\gamma} dx= 0.
 \end{equation*}
 If $u>0$ is a bounded $(\A, \B)$-harmonic function in $\Rn$, then $u\equiv \hbox{const}$.
\end{theorem}

\begin{ex}
Let $p^{-}_{ B_{2R_k}\setminus B_{R_k}}\geq n+\delta +1+\gamma$, where $\{R_k\}_{k=1}^{\infty}$  is some divergence increasing sequence and $\delta, \gamma >0$.
Assume that $\int_{\mathbb{R}^n} c(x) dx =0$ and $ c \not\equiv 0$, then the following equation has no positive bounded entire solutions:
    \begin{eqnarray*}
            - \div \A(x,u(x) \nabla u(x)) =c(x) u(x)^{\gamma}.
    \end{eqnarray*}

Indeed, if $u>0$ is a bounded solution to the above equation, then by
virtue of Theorem \ref{thm:bdd-grad-A-B} we get that
$u=const$. Thus, since $\A(x,t,0)=0$ we obtain
\begin{eqnarray*}
        \int_{\mathbb{R}^n} c(x)\phi(x) dx =0
\end{eqnarray*}
for all $\phi \in C_0^{\infty}(\mathbb{R}^n)$ and we get the contradiction with
$c \not\equiv 0$.

\end{ex}

\begin{proof}[Proof of Theorem~\ref{thm:bdd-grad-A-B}]
 Suppose that $|u|\leq M$ in $\Rn$. Using the Caccioppoli inequality~\eqref{Cac-ineq-positive} with $R=2R_k$ and $r=R_k$ we get for $R_k>1$:
\begin{align*}
 \!\!\int_{B_{R_k}}\!\!\!\!|\nabla u|^{p(x)}u^{\gamma-1}&\leq C(p^+, a_{-}, b^{+}, |\gamma|)
 \int_{B_{2R_k}\setminus B_{R_k}}\left(\frac{|u|}{R_k}\right)^{\px +\gamma -1}+\frac{2}{a_{-}|\gamma|}\int_{B_{2R_k}} \B(x, u, \nabla u) u^{\gamma}\eta^{p^{+}_{B_{2R_k}}}  \\
 &\leq C(p^+, a_{-}, b^{+}, |\gamma|)(M+1)^{p^+ + \gamma - 1}\int_{B_{2R_k}\setminus B_{R_k}}\frac{1}{R_k^{\px +\gamma -1}}+ \frac{2}{a_{-}|\gamma|}\int_{B_{2R_{k}}}\B(x, u, \nabla u) u^{\gamma}\eta^{p^{+}_{B_{2R_k}}} \\
 &\leq C(n, p^+, a_{-}, b^{+}, |\gamma|)(M+1)^{p^+ + \gamma-1}\omega(n)\frac{R_k^n}{R_k^{p^{-}_{ B_{2R_k}\setminus B_{R_k}} +\gamma -1}}+\frac{2}{a_{-}|\gamma|}\int_{B_{2R_{k}}}\!\!\B(x, u, \nabla u)u^{\gamma}\eta^{p^{+}_{B_{2R_k}}}.
\end{align*}
Let $k\rightarrow \infty$ in the above inequality to conclude that:
\begin{equation*}
 \int_{\mathbb{R}^n}|\nabla u|^{p(x)}u^{\gamma-1}=0.
\end{equation*}
Hence $\nabla u=0$ a.e. and $u=const$.
\end{proof}

\subsection*{Liouville theorems for inequalities of type $\div \A(x, u, \nabla u)\geq \B(x, u, \nabla u)$. Nonexistence results}

In this section we study the special cases of functions $\B$ for $(\A, \B)$-harmonic equations and inequalities as in \eqref{eq-sub-sol}, cf. Definition~\ref{A-harm}. First, we investigate $\B=f(u)$, then the Riccati-type inequalities with $\B=f(|\nabla u|)$. For both forms of $\B$ we obtain the Liouville type theorems and nonexistence results for supersolutions.

\smallskip
\noindent We begin with the case $\B=f(u)$.
\smallskip

\noindent Let $\A$ be a differential operator as in Definition~\ref{A-function} and $f:\R\to \R^{+}$ be a nonnegative continuous function. We say that $u\in W^{1,\p}_{loc}(\Rn)$ is a supersolution of $\div \A(x, u, \nabla u)=f(u)$ in $\R^n$ if
\begin{equation}\label{eq:Af}
-\int_{\Rn}\langle \A(x, u, \nabla u), \nabla \phi\rangle \geq \int_{\Rn} f(u)\phi,
\end{equation}
for all nonnegative test functions $\phi \in C_{0}^{\infty}(\Rn)$.

In the context of $\A$-harmonic equations with $p$-growth for $p=const$, differential inequality \eqref{eq:Af} has been studied e.g. by D'Ambrosio~\cite{A}, see Definition 3.1 and conditions (3.16)-(3.17).

\begin{ex}
 A typical function $f$ studied in \eqref{eq:Af} is $f(t)=ct^q$ for a constant $c$ and $q>0$, see e.g. D'Ambrosio--Mitidieri \cite{AM2, AM3}, Filippucci~\cite{Fil} and their references.
\end{ex}

\begin{theorem}\label{thm:Af}
 Let $p^{-}>n$. Suppose that $u$ is a bounded supersolution of \eqref{eq:Af} in $\Rn$ and that there exists exactly one $t\in \R$ such that $f(t)=0$. Then $u$ is constant.
\end{theorem}
\begin{ex}
 Let $f:\R\to \R^{+}$ be a nonnegative continuous function and assume additionally that $f$ is strictly convex or strictly monotone with one zero in $\R$. Then $f$ satisfies assumptions of Theorem~\ref{thm:Af}.
\end{ex}
\begin{rem}
 Theorem~\ref{thm:Af} partially extends Theorem 3.13 in D'Ambrosio~\cite{A} in the following sense. We assume $u$ to be bounded, not necessary $u\geq 0$ and allow $f$ to be a strictly convex function, not necessary in the form $f(u)=u^q$ as in \cite{A}. Also, for us $u$ need not be a solution (just a supersolution).

 Moreover, Theorem~\ref{thm:Af} corresponds to Theorem 3.1 in \cite{AM3} in the case $p>n$. There, $u$ is a nonnegative solution whereas we allow $u$ to be a bounded supersolution. For us $f\geq 0$, while condition (3.1) in \cite{AM3}
 allows $f$ to change sign.
\end{rem}
\begin{ex}
 We illustrate the discussion of Theorem~\ref{thm:Af} with
 an example showing that a positivity of $f$ is necessary condition for Theorem~\ref{thm:Af} to hold true.
 Indeed, let $u(x)=(b+|x|^{\frac{\px}{\px-1}})^{-1}$ for $b\geq 1$ and a Lipschitz continuous variable exponent $p$. Then $u$ is a bounded entire solution of the following inequality:
 \begin{equation}\label{ex4-ineq}
 \div(|\nabla u|^{\px-2}\nabla u)\geq
 \alpha(p^+,p^-)(n-2p^+)u^{2(p^+-1)}-\beta(p^+, p^-,\|\nabla
 p\|_{L^{\infty}(\Rn)}).
 \end{equation}
Direct computations show that
\[
\div(|\nabla u|^{\px-2}\nabla
u)=\left(\frac{\px}{\px-1}\right)^{\px-1}u^{2(\px-1)}\left\{n-2p(x)\frac{|x|^\frac{\px}{\px-1}}{b+|x|^\frac{\px}{\px-1}}
+\langle \nabla p, x\rangle
\ln\left(e\frac{|x|^{\frac{1}{\px-1}}}{(b+|x|^\frac{\px}{\px-1})^2}\right)\right\}.
\]
The constant $\alpha$ arises from the boundedness of the variable
exponent $p$, while $u^{2(p^+-1)}$ comes from the fact that $u\leq
1$ and hence $u^{2(\px-1)}\geq u^{2(p^+-1)}$. As for the constant
$\beta>0$, one needs to notice that the above logarithmic expression is
bounded.

Since in Theorem~\ref{thm:Af} it is assumed that $p^{-}>n$, the function on the right-hand side of \eqref{ex4-ineq}
is negative.
\end{ex}
\begin{proof}[Proof of Theorem~\ref{thm:Af}]
   By the definition of supersolutions, $u\in W^{1, \p}_{loc}(\Rn)$ and since $p^->n$, then the embedding theorem for variable exponent Sobolev spaces allows us to choose the H\"older continuous representative of $u$, see e.g. Diening--Harjulehto--H\"ast\"o--R\r u\v zi\v cka~\cite[Theorem 8.3.8]{DHHR}. We denote such a representative again by $u$.

   Let $R>1$ and $\phi\in C_{0}^{\infty}(\Rn)$ be nonnegative with support in a ball $B_{2R}$ and such that $\phi\equiv 1$ in $B_R$. Similarly to previous results we also assume that $|\nabla \phi|\leq \frac{c}{R}$. Using weak formulation \eqref{eq:Af}, growth assumptions on operator $\A$ (cf. Definition~\ref{A-function}) and the Young inequality we obtain the following estimate:
 \begin{align}
 \int_{B_R} f(u) &\leq \int_{B_{2R}} |\A(x, u, \nabla u)||\nabla \phi|\leq \int_{B_{2R}} b^{+}|\nabla u|^{\px-1}|\nabla \phi|\nonumber \\
 &\leq b^{+}\int_{B_{2R}} |\nabla u|^{\px}+\int_{B_{2R}\setminus B_{R}}|\nabla \phi|^{\px}. \label{eq:Af-2}
 \end{align}
 Since $f$ and $\phi$ are nonnegative, we have from \eqref{eq:Af} that $u$ is in particular a subsolution to $\A$-harmonic equation and thus $u$ satisfies Lemma~\ref{lem:Cac-ineq}. We apply it for $c=\inf_{\Rn}u$ and radii $2R$ and $3R$
 and use in \eqref{eq:Af-2} to arrive at the following inequality:
  \begin{align}
 \int_{B_R} f(u) &\leq  C(p^+, a_{-}, b^{+})\int_{B_{3R}\setminus B_{2R}}\frac{|u-\inf_{\Rn}u|^{\px}}{R^{\px}}+\int_{B_{2R}\setminus B_{R}}|\nabla \phi|^{\px}. \label{eq:Af-3}
 \end{align}
 Suppose that $u\leq M$. This and the decay condition on $|\nabla \phi|$ imply that
 \begin{align}\label{ineq:thm-Af}
 \int_{B_R} f(u) &\leq  C(p^+, a_{-}, b^{+})(1+M+|\inf_{\Rn}u|)^{p^{+}}\omega(n) R^{n-p^{-}}+c\omega(n)R^{n-p^{-}}.
 \end{align}
 Letting $R\to\infty$ we get that $\int_{\mathbb{R}^n} f(u)=0$, and thus $f(u(x))\equiv 0$ for all $x\in \Rn$. Since $f$ vanishes exactly at $t$, it holds that $u\equiv t$ and, hence, the proof is completed.
\end{proof}

From the proof of Theorem~\ref{thm:Af} we infer the following non-existence results for inequalities of type \eqref{eq:Af}.
\begin{cor}\label{cor-non-exist}
 Let $p^{-}>n$. Consider inequality \eqref{eq:Af} with $f$ strictly convex and such that $f(t)\not=0$ for all $t\in \R$. Then, there exists no bounded supersolution of \eqref{eq:Af}.
\end{cor}
\begin{proof}
 Suppose on the contrary, that there exists $M>0$ such that a supersolution $u$ satisfies $|u|\leq M$ in $\Rn$. Then we follow the steps of the proof for Theorem~\ref{thm:Af} and get that $f(u(x))\equiv 0$ for all $x\in \Rn$ leading to the contradiction with the assumption that $f\not =0$ in $\Rn$.
\end{proof}

As a matter of fact the stronger result holds.
\begin{cor}\label{cor-non-exist22}
 Let $p^{-}>n$. Consider inequality \eqref{eq:Af} with $f$ strictly increasing and $f(0)\not=0$. Then there exists no bounded supersolution of \eqref{eq:Af}.
\end{cor}
\begin{proof}
 As in the proof of Corollary~\ref{cor-non-exist} we proceed by assuming that a bounded supersolution exists and obtain that $f(u(x))\equiv 0$ for all $x\in \Rn$. Since $f$ cannot attain value zero, we reach the contradiction with the boundedness assumption and hence complete the proof of the corollary.
\end{proof}

\begin{rem}
 Corollaries \ref{cor-non-exist} and \ref{cor-non-exist22} correspond to Theorem 3.4 in D'Ambrosio--Mitidieri~\cite{AM3}. There, the fact that $f$ does not attain value $0$ is a necessary condition for the existence of entire solutions to $\div(|\nabla u|^{p-2}\nabla u)\geq f(u)$, under additional assumptions on $f$. See also \cite[Theorem 3.13]{AM3} for the setting of $\A$-harmonic supersolution.
\end{rem}

Similarly to Corollary~\ref{Ph-Lind} we obtain a variant of Theorem~\ref{thm:Af} where instead of the boundedness assumption on $u$ we impose the growth condition.

\begin{cor}\label{cor:Af}
 Let $p^{-}>n$. Assume that $f:\R^{+}\to \R^{+}$ is a nonnegative strictly convex continuous function. Suppose that $u\geq 0$ is a supersolution of \eqref{eq:Af} in $\Rn$ such that it satisfies
 \begin{equation*}
  |u(x)|\leq C|x|^{\alpha}\quad \hbox {for }\quad \alpha<\frac{p^-}{p^+}-\frac{n}{p^+}.
 \end{equation*}
 Then $u$ is constant.
\end{cor}
\begin{proof}
 The reasoning follows the steps of the proof of Theorem~\ref{thm:Af} with modifications related to substituting the supremum estimate of $|u|$ by the growth condition. In a consequence, the first integral on the right-hand side of \eqref{eq:Af-3} involves $\frac{1}{R^{p^-}}R^{\alpha p^++n}$ and instead of \eqref{ineq:thm-Af} we get:
 \begin{align*}
 \int_{B_R} f(u) &\leq  C(p^+, a_{-}, b^{+})\omega(n) R^{n-p^{-}+\alpha p^+}+c\omega(n)R^{n-p^{-}}.
 \end{align*}
 Letting $R\to\infty$ we get that under the assumptions on $p^-$ and $\alpha$ it holds that $\int_{B_R} f(u)=0$. Hence $f(u(x))\equiv 0$ for all $x\in \Rn$. Since $f$ is strictly convex and vanishes at $x_0$, it vanishes exactly at $x_0$. This implies that $u=const$ and the proof is completed.
\end{proof}
We will now discuss the case of supersolutions to the Riccati type equations. A typical example of such equations is
\[
\div \A(x, \nabla u)=a(x)|\nabla u|^q,
\]
where $\A$ is of $p$-harmonic type and $q$ is often assumed to satisfy some additional conditions, such as $p-1<q\leq p$ or $q>p$, see e.g. Martio~\cite{m} or Phuc~\cite{ph} for further discussion and references.

Let $f:\R_{+}\to \R$ be continuous and suppose that $\B=f(|\nabla u|)$ satisfies the following inequality for a given variable exponent $q\geq 1$ in $\Om$:
\begin{equation}\label{cond-f}
 c(x)|\nabla u|^{\qx} \leq f(|\nabla u|) \leq d(x)(1+|\nabla u|^{\px-1}),
\end{equation}
for all $x$ in the domain of $|\nabla u|$ and measurable functions $c$ and $d$ bounded, respectively, from below by $c_{-}>0$ and from above by $d_{+}>0$ with $c^{+}<d_{-}$.

Filippucci~\cite{Fil} studied the similar problems for $f(x,z,\xi)\geq a(x)z^q|\xi|^{\theta}$ with $q>0$ and special form of operators $\A$, cf. Theorem 3.1 in \cite{Fil}.
\begin{theorem}
 Let $p^{-}>n$ and let $f:\R_{+}\to \R$ be continuous and satisfies condition \eqref{cond-f} for some functions $c, d$ and a variable exponent $q$. Suppose that $u$ is a bounded function in  $W^{1,\p}_{loc}(\Rn)$ satisfying the following Riccati-type inequality
 \begin{equation*}
-\int_{\Rn}\langle \A(x, u, \nabla u), \nabla \phi\rangle \geq \int_{\Rn} f(|\nabla u|)\phi,
\end{equation*}
for all nonnegative test functions $\phi \in C_{0}^{\infty}(\Rn)$. Then $u$ is constant.
\end{theorem}

\begin{proof}
 We follow the proof of Theorem~\ref{thm:Af} and instead of estimate \eqref{ineq:thm-Af} we obtain
 \begin{align*}
 c_{-}\int_{B_R} |\nabla u|^{\qx}\leq \int_{B_R} f(|\nabla u|) &\leq  C(p^+, a_{-}, b^{+})(1+M+|\inf_{\Rn}u|)^{p^{+}}\omega(n) R^{n-p^{-}}+c\omega(n)R^{n-p^{-}}.
 \end{align*}
 We let $R\to \infty$ and get that $\int_{\Rn} |\nabla u|^{\qx}=0$. Hence $u=const$.
\end{proof}
\subsection*{Acknowledgements}
The authors would like to thank Richard Wheeden for discussion about the Poincar\'e inequality.

\end{document}